\title{Waiter-Client and Client-Waiter planarity, colorability and minor games}

\author{
Dan Hefetz
\thanks{School of Mathematics, University of Birmingham, Edgbaston, Birmingham B15 2TT, United
Kingdom. Email: d.hefetz@bham.ac.uk. Research supported by EPSRC grant EP/K033379/1.} 
\and Michael
Krivelevich \thanks{School of Mathematical Sciences, Raymond and
Beverly Sackler Faculty of Exact Sciences, Tel Aviv University,
6997801, Israel. Email: krivelev@post.tau.ac.il. Research supported in
part by USA-Israel BSF Grant 2010115 and by grant 912/12 from the
Israel Science Foundation.}
\and Wei En Tan
\thanks{School of Mathematics, University of Birmingham, Edgbaston, Birmingham B15 2TT, United
Kingdom. Email: WET916@bham.ac.uk.} 
}

\documentclass[12pt]{article}
\usepackage{amsmath,amsthm,amssymb,latexsym,color,epsfig,a4,enumerate}

\newtheorem{theorem}{Theorem} [section]

\newtheorem{proposition}[theorem]{Proposition}

\newtheorem{claim}[theorem]{Claim}
\newtheorem{corollary}[theorem]{Corollary}

\begin{document}

\maketitle

\begin{abstract}
For a finite set $X$, a family of sets ${\mathcal F} \subseteq 2^X$ and a positive integer $q$, we consider two types of two player, perfect information games with no chance moves. In each round of the $(1 : q)$ Waiter-Client game $(X, {\mathcal F})$, the first player, called Waiter, offers the second player, called Client, $q+1$ elements of the board $X$ which have not been offered previously. Client then chooses one of these elements which he claims and the remaining $q$ elements are claimed by Waiter. Waiter wins this game if by the time every element of $X$ has been claimed by some player, Client has claimed all elements of some $A \in {\mathcal F}$; otherwise Client is the winner. Client-Waiter games are defined analogously, the main difference being that Client wins the game if he manages to claim all elements of some $A \in {\mathcal F}$ and Waiter wins otherwise. In this paper we study the Waiter-Client and Client-Waiter versions of the non-planarity, $K_t$-minor and non-$k$-colorability games. For each such game, we give a fairly precise estimate of the unique integer $q$ at which the outcome of the game changes from Client's win to Waiter's win. We also discuss the relation between our results, random graphs, and the corresponding Maker-Breaker and Avoider-Enforcer games.  
\end{abstract}

\section{Introduction}\label{intro}

The theory of positional games on graphs and hypergraphs goes back to the seminal papers of Hales and Jewett~\cite{HJ}, of Lehman~\cite{Lehman} and of Erd\H{o}s and Selfridge~\cite{ES}. It has since become a highly developed area of combinatorics (see, e.g., the monograph of Beck~\cite{TTT} and the recent monograph~\cite{HKSS}). The most popular and widely studied positional games are the so-called Maker-Breaker games. Let $p$ and $q$ be positive integers, let $X$ be a finite set and let ${\mathcal F}$ be a family of subsets of $X$. In each round of the biased $(p : q)$ Maker-Breaker game $(X, {\mathcal F})$, Maker claims $p$ previously unclaimed elements of $X$ and then Breaker responds by claiming $q$ previously unclaimed elements of $X$. Maker wins this game if, by the time every element of $X$ has been claimed, he has claimed all elements of some set $A \in {\mathcal F}$; otherwise Breaker is the winner. The set $X$ is called the \emph{board} of the game, the elements of ${\mathcal F}$ are called the \emph{winning sets} and the integers $p$ and $q$ are Maker's bias and Breaker's bias, respectively. Since this is a finite, perfect information game with no chance moves and no possibility of a draw, one of the two players must have a winning strategy. Moreover, it is not hard to see that Maker-Breaker games are \emph{bias monotone}, that is, claiming more board elements than his bias specifies per round cannot harm that player. In particular, there exists a unique positive integer $b_{\mathcal F}$ such that Breaker has a winning strategy for the $(1 : q)$ game $(X, {\mathcal F})$ if and only if $q \geq b_{\mathcal F}$; we refer to this integer as the \emph{threshold bias} of the Maker-Breaker game $(X, {\mathcal F})$.  

The so-called Avoider-Enforcer games form another class of well-studied positional games. In such games, Enforcer aims to force Avoider to claim all elements of some set $A \in {\mathcal F}$. Avoider-Enforcer games are sometimes referred to as mis\`ere Maker-Breaker games. There are two different sets of rules for Avoider-Enforcer games: \emph{strict rules} under which the number of board elements a player claims per round is precisely his bias and \emph{monotone rules} under which the number of board elements a player claims per round is at least as large as his bias (for more information on Avoider-Enforcer games see, for example,~\cite{HKSae, HKSSae, HKSS}).

One major motivation for studying biased Maker-Breaker and Avoider-Enforcer games is their relation to the theory of random graphs via the so-called \emph{probabilistic intuition}. Consider, for example, a $(1 : q)$ Maker-Breaker game $(X, {\mathcal F})$, where $X = E(K_n)$ is the edge-set of the complete graph on $n$ vertices. The following heuristic argument, first employed by Chv\'atal and Erd\H{o}s in~\cite{CE}, can be used to predict the winner of this game. This heuristic suggests that the player who has a higher chance to win the game when both players are playing randomly is also the one who wins the game when both players are playing optimally. More precisely, if the random graph $G(n,m)$ with $m = \left\lceil \binom{n}{2}/(q+1) \right\rceil$ edges contains all edges of some $A \in {\mathcal F}$ with probability tending to $1$ as $n$ tends to infinity, then Maker has a winning strategy for $(E(K_n), {\mathcal F})$. If, on the other hand, this probability tends to $0$ as $n$ tends to infinity, then $(E(K_n), {\mathcal F})$ is Breaker's win. This is highly unexpected as, in any positional game, both players have deterministic optimal strategies. Moreover, in most natural games, playing randomly against an optimal opponent leads to very poor results. As noted above, this is just a heuristic and does not always predict the outcome of the game correctly. Nevertheless, the probabilistic intuition is remarkably useful. Natural examples where this heuristic predicts the winner correctly, include the connectivity game (that is, ${\mathcal F}$ consists of all connected subgraphs of $K_n$)~\cite{GSz} and the Hamiltonicity game (that is, ${\mathcal F}$ consists of all Hamiltonian subgraphs of $K_n$)~\cite{kriv}. On the other hand, it was proved in~\cite{BL} that the probabilistic intuition fails (though another probabilistic reasoning is in play here) for the $H$-game (that is, ${\mathcal F}$ consists of all copies of some fixed predetermined graph $H$ in $K_n$). The probabilistic intuition is also useful when analyzing biased Avoider-Enforcer games (especially under monotone rules). Some examples can be found in~\cite{HKSSpcm} and~\cite{HKSSae}.
  
In this paper, we study \emph{Waiter-Client} and \emph{Client-Waiter} positional games. Such games are closely related to Maker-Breaker and Avoider-Enforcer games; the main difference being the process of selecting board elements. In every round of the biased $(p : q)$ Waiter-Client game $(X, {\mathcal F})$, the first player, called Waiter, offers the second player, called Client, $p+q$ previously unclaimed elements of $X$. Client then chooses $p$ of these elements which he claims, and the remaining $q$ elements are claimed by Waiter. If, in the final round of the game, only $1 \leq t < p+q$ unclaimed elements remain, then Client chooses $\max \{0, t - q\}$ elements which he claims and the remaining $\min \{t, q\}$ elements are claimed by Waiter. The game ends as soon as all elements of $X$ have been claimed. Waiter wins this game if he manages to force Client to claim all elements of some $A \in {\mathcal F}$; otherwise Client is the winner. Client-Waiter games are defined analogously, the main difference being that Client wins if and only if he manages to claim all elements of some $A \in {\mathcal F}$ (otherwise Waiter is the winner). Additionally, there are two technical differences between Client-Waiter and Waiter-Client games. Firstly, in a Client-Waiter game, Waiter is allowed to offer less board elements per round than his bias specifies. More precisely, in every round of a $(p : q)$ Client-Waiter game, Waiter offers $t$ elements, where $p \leq t \leq p+q$. Client chooses $p$ of these, which he keeps, and the remaining $t - p$ elements are claimed by Waiter. Secondly, if there are $r < p+q$ free elements offered to Client in the final round of the game, he first claims $\min\{r,p\}$ of these and any remaining elements are claimed by Waiter. As with Maker-Breaker games, it is not hard to see that Waiter-Client games are monotone in Waiter's bias $q$. In particular, essentially any Waiter-Client game $(X, {\mathcal F})$ has a threshold bias, that is, a unique positive integer $b_{\mathcal F}$ such that Client has a winning strategy for the $(1 : q)$ Waiter-Client game $(X, {\mathcal F})$ if and only if $q \geq b_{\mathcal F}$. From the way that Client-Waiter games have been defined, it is obvious that they also have a threshold bias (as observed in~\cite{Bednarska}, this does not remain true if we require Waiter to offer exactly $p+q$ board elements per round). The threshold bias of the Client-Waiter game $(X, {\mathcal F})$ is the unique positive integer $b_{\mathcal F}$ such that Waiter has a winning strategy for the $(1 : q)$ game if and only if $q \geq b_{\mathcal F}$.     

Waiter-Client and Client-Waiter games were first defined and studied by Beck under the names Picker-Chooser and Chooser-Picker, respectively (see, e.g.,~\cite{becksec}). However, since \emph{picking} and \emph{choosing} are essentially the same, we feel that the names Waiter and Client, which first appeared in~\cite{BHL}, help the reader to distinguish more easily between the roles of the two players.

As with Maker-Breaker and Avoider-Enforcer games, the probabilistic intuition turns out to be useful for Waiter-Client games as well. In particular, it is known to hold for the $K_t$-game (and in fact, for many other fixed graph games)~\cite{BHL}, for the diameter two game (that is, the winning sets are the edge-sets of all subgraphs of $K_n$ with diameter at most two)~\cite{ctcs} and for the giant component game (that is, the game on $E(K_n)$ in which Waiter tries to force Client to build a connected component on as many vertices as possible)~\cite{BHKL}.

This paper is devoted to the study of several natural $(1 : q)$ Waiter-Client and Client-Waiter games, played on $E(K_n)$. For both the Waiter-Client and Client-Waiter versions, we will study the non-planarity game $(E(K_n), \mathcal{NP})$, where $\mathcal{NP}$ consists of all non-planar subgraphs of $K_n$, the $K_t$-minor game $(E(K_n), \mathcal{M}_t)$, where $\mathcal{M}_t$ consists of all subgraphs of $K_n$ that admit a $K_t$-minor, and the non-$k$-colorability game $(E(K_n), \mathcal{NC}_k)$, where $\mathcal{NC}_k$ consists of all non-$k$-colorable subgraphs of $K_n$. The analogous Maker-Breaker and Avoider-Enforcer games were studied in~\cite{HKSSpcm}.

It was proved in~\cite{BPcycles} that, if $q \geq n/2$, then when playing a $(1 : q)$ Maker-Breaker game on $E(K_n)$, Breaker can force Maker to build a forest, that is, Breaker has a winning strategy for the $(1 : q)$ game $\mathcal{M}_t$, for every $t \geq 3$. On the other hand, it was proved in~\cite{HKSSpcm} that, for every fixed $\varepsilon > 0$, there exists a constant $c = c(\varepsilon) > 0$ such that, if $q \leq (1/2 - \varepsilon) n$, then Maker has a winning strategy for the $(1 : q)$ Maker-Breaker game $\mathcal{M}_t$ for every $t \leq c \sqrt{n/\log n}$. The strict Avoider-Enforcer minor game was considered in~\cite{HKSSpcm} as well. It was proved there that, if $q \leq (1/2 - \varepsilon) n$, then Enforcer has a winning strategy for the $(1 : q)$ game $\mathcal{M}_t$ for every $t \leq n^{\gamma}$, where $\gamma$ is a function of $\varepsilon$. On the other hand, improving a result from~\cite{HKSSpcm}, it was recently proved in~\cite{CEPT} that, if $q \geq 200 n \log n$, then Avoider has a winning strategy for the $(1 : q)$ game $\mathcal{M}_t$, for every $t \geq 4$.     

As with Maker-Breaker games, we are able to determine the asymptotic value of the threshold bias of the Waiter-Client minor game (it is twice as large as the corresponding threshold of Maker-Breaker games). Moreover, in contrast to the aforementioned results for Maker-Breaker games, the accuracy of our results increases as the order of the minor Waiter aims to force in Client's graph decreases. Additionally, we can prove that, even when playing with a bias which is arbitrarily close to $n$, Waiter can force Client to build a $K_t$-minor for $t = \Omega(\sqrt{n})$. This order of magnitude is best possible as, at the end of the game, Client's graph will contain $O(n)$ edges.      

\begin{theorem} \label{th::WCminor}
Let $n$ be a sufficiently large integer, let $\varepsilon = \varepsilon(n) \geq 4 n^{- 1/4}$ and let $t=t(n) \leq \varepsilon^2 \sqrt{n}/5$ be an integer. If $q \leq (1 - \varepsilon) n$, then when playing a $(1 : q)$ Waiter-Client game on $E(K_n)$, Waiter has a strategy to force Client to build a graph which admits a $K_t$-minor. On the other hand, if $q \geq n + \eta$, where $\eta = \eta(n) \geq n^{2/3} \log n$, then when playing a $(1 : q)$ Waiter-Client game on $E(K_n)$, Client can ensure that his graph will be $K_4$-minor free throughout the game.  
\end{theorem}

Theorem~\ref{th::WCminor} exhibits a very strong probabilistic intuition; stronger than the aforementioned corresponding results for Maker-Breaker and Avoider-Enforcer games in several respects. Indeed, it is well known (see, e.g.,~\cite{Bol, JLR}) that if $p \leq (1 - \varepsilon)/n$ for an arbitrarily small but fixed $\varepsilon > 0$, then asymptotically almost surely (or a.a.s. for brevity) every connected component of the random graph $G(n,p)$\footnote{Previously, when first introducing the notion of probabilistic intuition, we used the uniform random graph model $G(n,m)$ which clearly highlights that Client's graph and the random graph it is compared to, have the same number of edges. However, for convenience, we will work with the much more commonly used binomial random graph model $G(n,p)$ in the remainder of this paper. It is known that, for our purposes, both models are equivalent (see, e.g.,~\cite{JLR}).} contains at most one cycle and thus $G(n,p)$ is $K_4$-minor free. On the other hand, it was proved in~\cite{FKO} that if $p \geq (1 + \varepsilon)/n$ for an arbitrarily small but fixed $\varepsilon > 0$, then a.a.s. the random graph $G(n,p)$ admits a $K_t$-minor for $t = \Theta(\sqrt{n})$ (see also~\cite{BCE, KS} for earlier results).

Theorem~\ref{th::WCminor} shows that, for every $4 \leq t = O(\sqrt{n})$, the threshold bias of the Waiter-Client $K_t$-minor game is $(1 + o(1)) n$. Since the threshold bias is a unique integer, it is natural to wonder whether it is precisely $n$. The following result shows that, at least for large $t$, this is not the case. 

\begin{theorem} \label{th::WCminor2}
Let $n, t$ and $\alpha$ be positive integers where $n$ is sufficiently large, $t = t(n) \geq \frac{C \log \log n}{\log \log \log n}$ for some sufficiently large constant $C$, $0 \leq \alpha < c t \log t$ for some sufficiently small constant $c > 0$ and $\alpha = o(n)$. Then Client has a winning strategy for the $(1 : q)$ Waiter-Client game $(E(K_n), \mathcal{M}_t)$ for every $q \geq n - \alpha$.
\end{theorem}

For the Client-Waiter minor game, we can again determine the asymptotic value of the threshold bias --- this time it coincides with the threshold bias of the corresponding Maker-Breaker game. 

\begin{theorem} \label{th::CWminor}
Let $n$ be a sufficiently large integer and let $\varepsilon = \varepsilon(n) > 0$. If $q \geq n/2 - 1$, then when playing a $(1 : q)$ Client-Waiter game on $E(K_n)$, Waiter has a strategy to keep Client's graph $K_3$-minor free throughout the game. On the other hand, if $q \leq (1/2 - \varepsilon) n$, then when playing a $(1 : q)$ Client-Waiter game on $E(K_n)$, Client has a strategy to build a graph which admits a $K_t$-minor for $t = (\varepsilon n)^{c \varepsilon}$, where $c > 0$ is an absolute constant.
\end{theorem}

As simple corollaries of Theorems~\ref{th::WCminor} and~\ref{th::CWminor}, we determine the asymptotic value of the threshold bias of the Waiter-Client and Client-Waiter non-planarity games, respectively.

\begin{corollary} \label{cor::WCplanarity}
Let $n$ be a sufficiently large integer and let $\varepsilon = \varepsilon(n) \geq 5 n^{- 1/4}$. If $q \leq (1 - \varepsilon) n$, then when playing a $(1 : q)$ Waiter-Client game on $E(K_n)$, Waiter has a strategy to force Client to build a non-planar graph. On the other hand, if $q \geq n + \eta$, where $\eta = \eta(n) \geq n^{2/3} \log n$, then when playing a $(1 : q)$ Waiter-Client game on $E(K_n)$, Client can ensure that his graph will remain planar throughout the game.
\end{corollary}

\begin{corollary} \label{cor::CWplanarity}
Let $n$ be a sufficiently large integer. If $q \geq n/2 - 1$, then when playing a $(1 : q)$ Client-Waiter game on $E(K_n)$, Waiter has a strategy to keep Client's graph planar throughout the game. On the other hand, there exists a constant $c > 0$ such that if $q \leq n/2 - c n/\log n$, then when playing a $(1 : q)$ Client-Waiter game on $E(K_n)$, Client has a strategy to build a non-planar graph.
\end{corollary}

It was proved in~\cite{HKSSpcm} that there exist absolute constants $c_1 \geq c_2 > 0$ such that the threshold bias of the Maker-Breaker non-$k$-colorability game $\mathcal{NC}_k$ is between $c_1 n/(k \log k)$ and $c_2 n/(k \log k)$; when $k$ tends to infinity, one can take $c_1 = 2 + o(1)$ and $c_2 = \log 2/2 - o(1)$. Up to the values of the constants $c_1$ and $c_2$, this matches the probabilistic intuition, as the threshold probability for the non-$k$-colorability of the random graph $G(n,p)$ is about $p = (2 k \log k)/n$~\cite{AN}. It was also proved in~\cite{HKSSpcm} that Enforcer has a winning strategy for the $(1 : q)$ strict Avoider-Enforcer game $(E(K_n), \mathcal{NC}_k)$, whenever $q \leq c n/ (k \log k)$ for an appropriate absolute constant $c > 0$. On the other hand, improving a result from~\cite{HKSSpcm}, it was recently proved in~\cite{CEPT} that, if $q \geq 200 n \log n$, then Avoider has a strategy to keep his graph $3$-colorable. We will prove that the Waiter-Client and Client-Waiter non-$k$-colorability games behave similarly to the corresponding Maker-Breaker games. In particular, they exhibit a similar probabilistic intuition.    

\begin{theorem} \label{th::kColorabilityWC}
Let $k \geq 2$ be a fixed integer and let $n$ be a sufficiently large integer. Then there exists a function $\alpha = \alpha(k) > 0$ which tends to $0$ as $k$ tends to infinity such that the following holds. If $q \geq (8e + \alpha) n/(k \log k)$, then when playing a $(1 : q)$ Waiter-Client game on $E(K_n)$, Client can ensure that his graph will be $k$-colorable throughout the game. On the other hand, if $q \leq (\log 2/4 - \alpha) n/(k \log k)$, then when playing a $(1 : q)$ Waiter-Client game on $E(K_n)$, Waiter can force Client to build a non-$k$-colorable graph. 
\end{theorem}

\begin{theorem} \label{th::kColorabilityCW}
Let $k \geq 2$ be a fixed integer and let $n$ be a sufficiently large integer. Then there exists a function $\alpha = \alpha(k) > 0$ which tends to $0$ as $k$ tends to infinity such that the following holds. If $q \geq (4 + \alpha) n/(k \log k)$, then when playing a $(1 : q)$ Client-Waiter game on $E(K_n)$, Waiter has a strategy to ensure that Client builds a $k$-colorable graph. On the other hand, if $q \leq (\log 2/2 - \alpha) n/(k \log k)$, then when playing a $(1 : q)$ Client-Waiter game on $E(K_n)$, Client has a strategy to build a non-$k$-colorable graph. 
\end{theorem}

\section{Preliminaries}

\noindent For the sake of simplicity and clarity of presentation, we do not make a par\-ti\-cu\-lar effort to optimize the constants obtained in some of our proofs. We also omit floor and ceiling signs whenever these are not crucial. Most of our results are asymptotic in nature and whenever necessary we assume that the number of vertices $n$ is sufficiently large. Throughout this paper, $\log$ stands for the natural logarithm, unless explicitly stated otherwise. Our graph-theoretic notation is standard and follows that of~\cite{West}. In particular, we use the following.

For a graph $G$, let $V(G)$ and $E(G)$ denote its sets of vertices and edges respectively, and let $v(G) = |V(G)|$ and $e(G) = |E(G)|$. For a set $A \subseteq V(G)$, let $E_G(A)$ denote the set of edges of $G$ with both endpoints in $A$ and let $e_G(A) = |E_G(A)|$. For disjoint sets $A,B \subseteq V(G)$, let $E_G(A,B)$ denote the set of edges of $G$ with one endpoint in $A$ and one endpoint in $B$, and let $e_G(A,B) = |E_G(A,B)|$. For a set $S \subseteq V(G)$, let $G[S]$ denote the subgraph of $G$ which is induced on the set $S$. The \emph{degree} of a vertex $u \in V(G)$ in $G$, denoted by $d_G(u)$, is the number of edges $e \in E(G)$ such that $u \in e$. The \emph{maximum degree} of a graph $G$ is $\Delta(G) = \max \{d_G(u) : u \in V(G)\}$ and the \emph{minimum degree} of a graph $G$ is $\delta(G) = \min \{d_G(u) : u \in V(G)\}$. Often, when there is no risk of confusion, we omit the subscript $G$ from the notation above. 

A graph is called a \emph{linear forest} if each of its connected components is a path. The \emph{girth} of a graph $G$ is the number of edges in a shortest cycle of $G$ (if $G$ is a forest, then its girth is infinity). A set $A \subseteq V(G)$ is said to be \emph{independent} in $G$ if $E_G(A) = \emptyset$. The \emph{independence number} of a graph $G$, denoted by $\alpha(G)$, is the maximum size of an independent set in $G$. The \emph{clique number} of a graph $G$, denoted by $\omega(G)$, is the size of a largest set $A \subseteq V(G)$ such that $uv \in E(G)$ for every two vertices $u,v \in A$. The \emph{chromatic number} of a graph $G$, denoted by $\chi(G)$, is the smallest integer $k$ for which $V(G)$ can be partitioned into $k$ independent sets. For a positive integer $t$ and a graph $G$, we say that \emph{$G$ admits a $K_t$-minor} if, for every $1 \leq i \leq t$, there exists a set $B_i \subseteq V(G)$ such that the following three properties hold:
\begin{description}
\item [(i)] $G[B_i]$ is connected for every $1 \leq i \leq t$.
\item [(ii)] $B_i \cap B_j = \emptyset$ for every $1 \leq i < j \leq t$.
\item [(iii)] $E_G(B_i, B_j) \neq \emptyset$ for every $1 \leq i < j \leq t$.
\end{description}   
Note that a graph is $K_3$-minor free if and only if it is a forest.     

Let $X$ be a finite set and let ${\mathcal F}$ be a family of subsets of $X$. The \emph{transversal} family of ${\mathcal F}$ is ${\mathcal F}^* := \{A \subseteq X : A \cap B \neq \emptyset \textrm{ for every } B \in {\mathcal F}\}$.  

Assume that some Waiter-Client or Client-Waiter game, played on the edge-set of some graph $H = (V,E)$, is in progress (in some of our arguments, we will consider games played on graphs other than $K_n$). At any given moment during this game, let $E_W$ denote the set of all edges that were claimed by Waiter up to that moment, let $E_C$ denote the set of all edges that were claimed by Client up to that moment, let $G_W = (V, E_W)$ and let $G_C = (V, E_C)$. Moreover, let $G_F = (V, E_F)$, where $E_F = E \setminus (E_W \cup E_C)$; the edges of $E_F$ are called \emph{free}.   

\bigskip

Next, we state four game-theoretic results which will be used repeatedly in this paper. The first such result is due to Beck~\cite{TTT}. He observed that a straightforward adaptation of his proof of a sufficient condition for Breaker's win, which is based on the well-known potential method of Erd\H{o}s and Selfridge~\cite{ES}, yields the following.

\begin{theorem}[implicit in~\cite{TTT}] \label{th::ClientBES} 
Let $q$ be a positive integer, let $X$ be a finite set, let ${\mathcal F}$ be a family of (not necessarily distinct) subsets of $X$ and let $\Phi(\mathcal{F}) = \sum_{A \in \mathcal{F}} (q+1)^{-|A|}$. Then, when playing the $(1 : q)$ Waiter-Client game $(X, {\mathcal F})$, Client has a strategy to avoid fully claiming more than $\Phi(\mathcal{F})$ sets in ${\mathcal F}$.  
\end{theorem}

We include a short proof which is based on a random strategy.

\noindent \emph{Proof of Theorem~\ref{th::ClientBES}}.
Fix an arbitrary strategy ${\mathcal S}_W$ of Waiter. Client plays randomly, i.e., in each round he claims one of the elements he is offered uniformly at random and independently of all other choices. Let $Y_C$ denote the set of board elements Client has claimed throughout the game. Fix an arbitrary set $A \in {\mathcal F}$. If, in some round, Waiter offers Client at least two elements of $A$, then surely Client will not fully claim $A$. Hence, if $A \subseteq Y_C$, then Waiter must have offered Client the elements of $A$ in $|A|$ different rounds, one element per round. Moreover, in each such round he must have offered precisely $q+1$ elements, since, if Waiter offers less (recall that this might happen in the last round), then he claims all of them. Even though the elements Waiter offers in some round may depend on the previous choices made by Client, in every round in which Waiter offered Client an element of $A$, Client claimed this (unique) element with probability $1/(q+1)$. Since Client's choices are independent, it follows that $Pr(A \subseteq Y_C) \leq (q+1)^{- |A|}$. Since $A$ was arbitrary we have $\mathbb{E}(|\{A \in {\mathcal F} : A \subseteq Y_C\}|) \leq \Phi(\mathcal{F})$. It follows that, if Waiter plays according to ${\mathcal S}_W$ and Client plays randomly, then with positive probability Client can avoid fully claiming more than $\Phi(\mathcal{F})$ sets in ${\mathcal F}$. Since ${\mathcal S}_W$ was arbitrary, this is true for any fixed strategy of Waiter. Hence, there is no strategy for Waiter to force Client to fully claim more than $\Phi(\mathcal{F})$ sets in ${\mathcal F}$. Since this is a finite, perfect information game with no chance moves, we conclude that there exists a strategy for Client to avoid fully claiming more than $\Phi(\mathcal{F})$ sets in ${\mathcal F}$.   
{\hfill $\Box$ \medskip\\}

When proving Theorem~\ref{th::CWminor} we will want to use a criterion similar to the one stated in Theorem~\ref{th::ClientBES}. However, since in Client-Waiter games Waiter is allowed to offer less board elements per round than his bias specifies, we cannot expect Client to avoid fully claiming certain sets. The following result of O. Dean~\cite{Dean} overcomes this problem.

\begin{theorem} [\cite{Dean}] \label{th::Oren}
Let $q$ be a positive integer, let $X$ be a finite set, let ${\mathcal F}$ be a family of (not necessarily distinct) subsets of $X$ and let $\Phi(\mathcal{F}) = \sum_{A \in \mathcal{F}} (q+1)^{-|A|}$. Then, when playing the $(1 : q)$ Client-Waiter game $(X, {\mathcal F})$, Client has a strategy to claim the elements of a set $X_C \subseteq X$ of size $|X_C| \geq \lfloor |X|/(q+1) \rfloor$ which fully contains at most $2 \Phi(\mathcal{F})$ sets in ${\mathcal F}$.
\end{theorem}

For the sake of completeness we include Dean's proof.    

\noindent \emph{Proof of Theorem~\ref{th::Oren}}.
For every positive integer $i$, let $W_i$ denote the set of board elements Waiter offers Client in round $i$ and let $\alpha_i = |W_i|/(q+1)$. Client plays randomly, i.e., for every $i$, in the $i$th round he claims an element of $W_i$ uniformly at random and independently of all other choices. Moreover, for every $i$, in the $i$th round he places his chosen element in $X_C$ with probability $\alpha_i$. 

Fix an arbitrary $x \in X$ and let $i$ denote the unique integer for which $x \in W_i$. Then $Pr(x \in X_C) = 1/|W_i| \cdot \alpha_i = 1/(q+1)$. Now, fix an arbitrary set $A \in {\mathcal F}$. If, in some round, Waiter offers Client at least two elements of $A$, then surely Client will not fully claim $A$. Hence, as in our proof of Theorem~\ref{th::ClientBES}, $Pr(A \subseteq X_C) \leq (q+1)^{- |A|}$ and thus $\mathbb{E}(|\{A \in {\mathcal F} : A \subseteq X_C\}|) \leq \Phi(\mathcal{F})$. It then follows by Markov's inequality that 
\begin{equation} \label{eq::fewWinningSets}
Pr(|\{A \in {\mathcal F} : A \subseteq X_C\}| > 2 \Phi(\mathcal{F})) < 1/2.
\end{equation}

Let $m$ denote the total number of rounds played in the game. Note that $|X_C| = \sum_{i=1}^m Z_i$, where $Z_1, \ldots, Z_m$ are independent Bernoulli random variables with $Pr(Z_i = 1) = \alpha_i$ for every $1 \leq i \leq m$. In particular, 
$$
\mathbb{E}(|X_C|) = \sum_{i=1}^m \mathbb{E}(Z_i) = \sum_{i=1}^m \alpha_i = |X|/(q+1).  
$$
Hence
\begin{equation} \label{eq::largeChosenSet}
Pr(|X_C| \geq \lfloor |X|/(q+1) \rfloor) \geq Pr(Bin(|X|, 1/(q+1)) \geq \lfloor |X|/(q+1) \rfloor) \geq 1/2,
\end{equation} 
where the first inequality holds by Theorem 5 from~\cite{Hoeffding}.

Combining~\eqref{eq::fewWinningSets} and~\eqref{eq::largeChosenSet} we conclude that there exists a strategy for Client to ensure that both $|X_C| \geq \lfloor |X|/(q+1) \rfloor$ and $|\{A \in {\mathcal F} : A \subseteq X_C\}| \leq 2 \Phi(\mathcal{F})$ will hold at the end of the game. 
{\hfill $\Box$ \medskip\\}

The third result was proved in~\cite{HKT}. It is a strengthening of another result of Beck~\cite{TTT}. 

\begin{theorem}[\cite{HKT}] \label{th::ClientBESpreventing} 
Let $q$ be a positive integer, let $X$ be a finite set and let $\mathcal{F}$ be a family of subsets of $X$. If
$$
\sum_{A \in \mathcal{F}} \left(\frac{q}{q+1}\right)^{|A|} < 1,
$$
then Client has a winning strategy for the $(1:q)$ Client-Waiter game $(X, \mathcal{F}^*)$.
\end{theorem}

The fourth result is a rephrased version of Corollary 1.5 from~\cite{Bednarska}.

\begin{theorem}[\cite{Bednarska}] \label{th::WaiterBES}
Let $q$ be a positive integer, let $X$ be a finite set and let ${\mathcal F}$ be a family of subsets of $X$. If
$$
\sum_{A \in {\mathcal F}} 2^{-|A|/(2q-1)} < 1/2 \,,
$$
then Waiter has a winning strategy for the $(1 : q)$ Waiter-Client game $(X, {\mathcal F}^*)$.
\end{theorem} 

The rest of this paper is organized as follows: in Section~\ref{sec::minorPlanarity} we prove Theorems~\ref{th::WCminor}, \ref{th::WCminor2} and~\ref{th::CWminor} as well as Corollaries~\ref{cor::WCplanarity} and~\ref{cor::CWplanarity}. In Section~\ref{sec::colorability} we prove Theorems~\ref{th::kColorabilityWC} and~\ref{th::kColorabilityCW}. Finally, in Section~\ref{sec::openprob} we present some open problems and conjectures.

\section{Minor and non-planarity games} \label{sec::minorPlanarity}

In this section we first prove Theorems~\ref{th::WCminor}, \ref{th::WCminor2} and~\ref{th::CWminor} and then deduce Corollaries~\ref{cor::WCplanarity} and~\ref{cor::CWplanarity} as straightforward consequences.  

\bigskip

\noindent \textbf{Proof of Theorem~\ref{th::WCminor}} 
Let $n$ be sufficiently large and let $q \leq (1 - \varepsilon) n$. We describe a strategy for Waiter to force a $K_t$-minor in Client's graph; it is divided into the following three stages (see also Figure 1): 

\bigskip

\noindent \textbf{Stage I:} Let $A \subseteq V(K_n)$ be an arbitrary set of size $\varepsilon n/2$ and let $B = V(K_n) \setminus A$. Offering only edges of $E_{K_n}(B)$ (in a manner that we describe in detail later), Waiter forces Client to build a path $P$ on at least $\varepsilon n/2$ vertices.    

\medskip

\noindent \textbf{Stage II:} Offering only edges of $E_{K_n}(A, V(P))$ (in a manner that we describe in detail later), Waiter forces Client to build a matching $M$ of size at least $\varepsilon^2 n/5$.

\medskip

\noindent \textbf{Stage III:} Let $P$ be split into $t$ consecutive vertex disjoint paths $P_1, \ldots, P_t$, each containing at least $\lfloor \sqrt{n} \rfloor$ endpoints of the matching $M$. For every $1 \leq i \leq t$, let $D_i = \{u \in A : \exists v \in V(P_i) \textrm{ such that } uv \in M\}$. For as long as there exist indices $1 \leq i < j \leq t$ such that $E_{G_C}(D_i, D_j) = \emptyset$, Waiter chooses such indices arbitrarily and offers Client $q+1$ arbitrary edges of $E_{K_n}(D_i, D_j)$. Once $E_{G_C}(D_i, D_j) \neq \emptyset$ for all $1 \leq i < j \leq t$, Waiter plays arbitrarily until the end of the game.   

\bigskip

\begin{figure} 
\begin{center}
\includegraphics[scale=0.85]{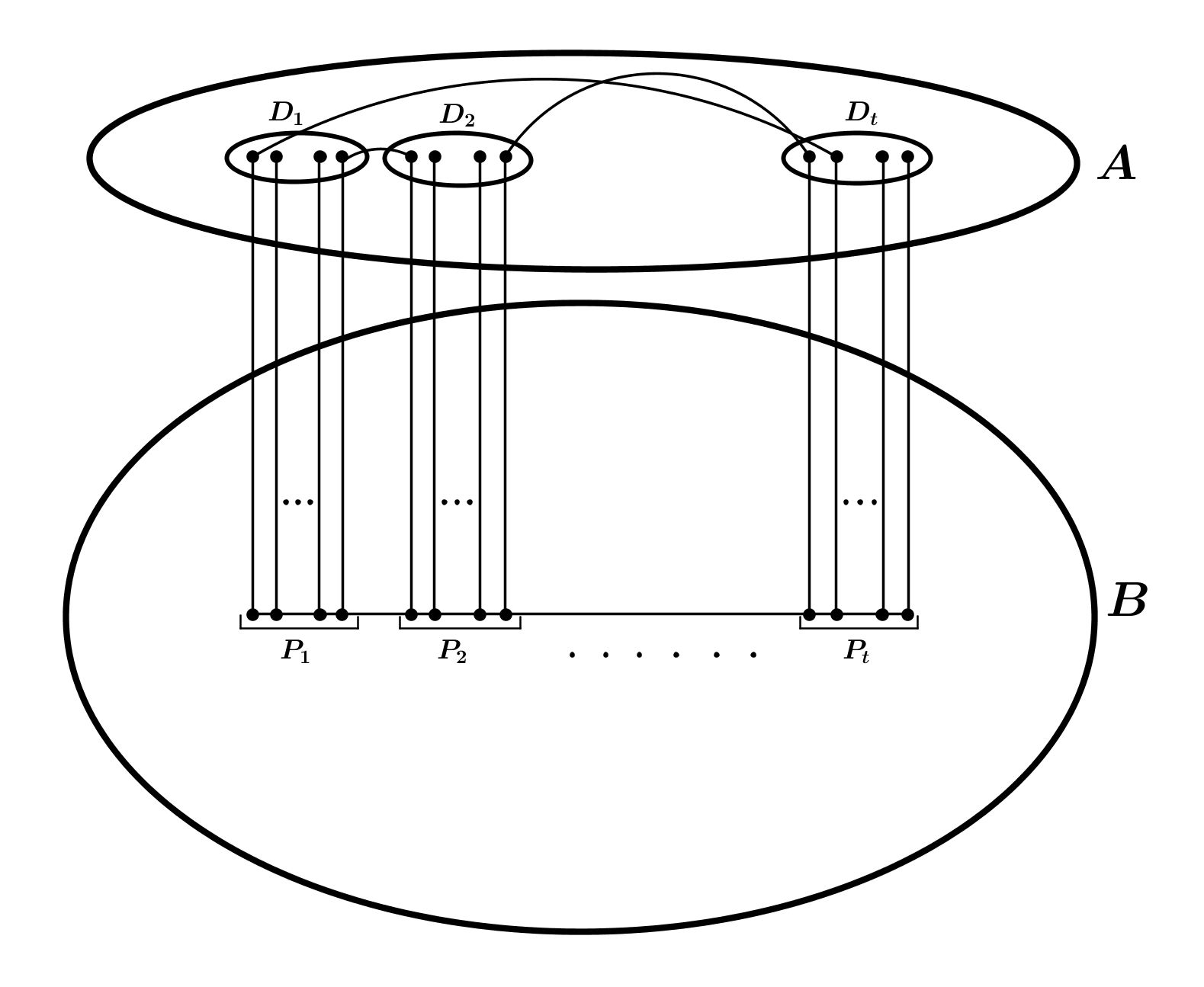}
\end{center}
\caption{An illustration of the graph Waiter is trying to force Client into building.}
\end{figure}

\bigskip

Assuming that Waiter can follow the proposed strategy, by contracting every edge with both endpoints in $V(P_i) \cup D_i$ for every $1 \leq i \leq t$, we obtain a $K_t$. Hence, Client's graph admits a $K_t$-minor as claimed. It thus remains to prove that Waiter can indeed play according to the proposed strategy; we do so for each stage separately.  

\medskip

Since $|B| - q \geq \varepsilon n/2$, the fact that Waiter can follow Stage I of the proposed strategy is an immediate corollary of the following claim.

\begin{claim} \label{cl::longPath}
Playing a $(1 : q)$ Waiter-Client game on $E(K_m)$, Waiter can force Client to build a path on at least $m - q$ vertices.
\end{claim}

\begin{proof}
We describe a strategy for Waiter. For every positive integer $i$, let $P_i$ denote Client's path immediately before the $i$th round; in particular, let $P_1 = x_1$, where $x_1 \in V(K_m)$ is an arbitrary vertex. Assume that $P_i = (x_1, \ldots, x_i)$ holds for some $i \geq 1$. If possible, in the $i$th round, Waiter offers Client the edges of $\{x_i y_j : 1 \leq j \leq q+1\}$, where $y_1, \ldots, y_{q+1}$ are arbitrary vertices of $V(K_m) \setminus V(P_i)$. By claiming any one of these edges, Client creates a path $P_{i+1} = (x_1, \ldots, x_i, x_{i+1})$. Once this is no longer possible, we must have $m - i = |V(K_m) \setminus V(P_i)| < q+1$, entailing $i \geq m - q$.
\end{proof} 

Next, we prove that Waiter can follow Stage II of the proposed strategy. Note that, by the description of Stage I, all edges of $E_{K_n}(A, V(P))$ are free at the beginning of Stage II. For as long as possible, in each round of this stage, Waiter offers Client $q+1$ arbitrary edges, which are disjoint from any edge Client has previously claimed in Stage II. It is thus evident that the graph Client builds in this stage is a matching; it remains to prove that it contains at least $\varepsilon^2 n/5$ edges. Suppose for a contradiction that, when following this strategy, Waiter can only force a matching of size $r < \varepsilon^2 n/5$ in Client's graph. Since Waiter cannot further enlarge Client's matching, it follows that he does not have enough edges to offer in accordance with Stage II of the proposed strategy. In particular, 
\begin{equation} \label{eq::smallMatching}
(\varepsilon n/2 - r)^2 - q r < q+1 \,.
\end{equation}

\noindent However, by the assumed lower bound on $\varepsilon$ we have
$$
(\varepsilon n/2 - r)^2 - q r > (\varepsilon n/2 - \varepsilon^2 n/5)^2 - \varepsilon^2 n^2 (1 - \varepsilon)/5 \geq \varepsilon^2 n^2/20 \geq (1 - \varepsilon) n + 1 \geq q+1 \,,
$$  
contrary to~\eqref{eq::smallMatching}. 

\medskip

Finally, we prove that Waiter can play according to Stage III of the proposed strategy. It follows by Stage II and by the assumed upper bound on $t$ that $|M| \geq \varepsilon^2 n/5 \geq t \sqrt{n}$. Therefore, $P$ can indeed be split into $t$ consecutive vertex disjoint paths $P_1, \ldots, P_t$, each containing at least $\lfloor \sqrt{n} \rfloor$ endpoints of $M$. By definition, $|D_i| \geq \lfloor \sqrt{n} \rfloor$ holds for every $1 \leq i \leq t$. Therefore, $|D_i||D_j| \geq (1 - \varepsilon) n + 1 \geq q+1$ holds for all $1 \leq i < j \leq t$. Since, by the description of Stages I and II, all edges of $E_{K_n}(D_i, D_j)$ are free at the beginning of Stage III, it follows that Waiter can ensure that Client will claim an edge of $E_{K_n}(D_i, D_j)$ for all $1 \leq i < j \leq t$.

Next, assume that $q \geq n + \eta$. Let ${\mathcal F}_1$ denote the family of edge-sets of cycles of $K_n$ of length at least $\sqrt[3]{n}/2$. Then
\begin{eqnarray} \label{eq::phi1}
\Phi({\mathcal F}_1) &=& \sum_{A \in {\mathcal F}_1} (q+1)^{-|A|} = \sum_{k = \sqrt[3]{n}/2}^n \binom{n}{k} \frac{(k-1)!}{2} (q+1)^{-k} < \sum_{k = \sqrt[3]{n}/2}^{\infty} \frac{1}{k} \left(\frac{n}{q}\right)^k \nonumber \\ 
&<& \left(\frac{n}{q} \right)^{\sqrt[3]{n}/2 - 1} \sum_{k=1}^{\infty} \frac{1}{k} \left(\frac{n}{q}\right)^k = \left(\frac{n}{q} \right)^{\sqrt[3]{n}/2 - 1} \log \left(\frac{q}{q-n} \right) \nonumber \\
&\leq& \left(\frac{n}{n + n^{2/3} \log n} \right)^{\sqrt[3]{n}/2 - 1} \log \left(\frac{n + n^{2/3} \log n}{n^{2/3} \log n} \right) \leq \exp \left\{- \frac{(\sqrt[3]{n}/2 - 1) n^{2/3} \log n}{n + n^{2/3} \log n} \right\} \cdot \log n \nonumber \\ 
&=& o(1) \,,
\end{eqnarray}
where the third equality follows from the Taylor expansion $- \log (1-x) = \sum_{k=1}^{\infty} x^k/k$.

\medskip

Let ${\mathcal F}_2$ denote the family of edge-sets of all pairs of cycles $(C_1, C_2)$ of $K_n$, such that $|C_1| = \ell_1$, $|C_2| = \ell_2$, $\ell_2 \leq \ell_1 \leq \sqrt[3]{n}/2$, and $C_1 \cap C_2$ is a path on $s \geq 1$ vertices. Then 
\begin{eqnarray} \label{eq::phi2}
\Phi({\mathcal F}_2) &=& \sum_{A \in {\mathcal F}_2} (q+1)^{-|A|} \leq \sum_{\ell_1 = 3}^{\sqrt[3]{n}/2} \sum_{\ell_2 = 3}^{\ell_1} \sum_{s=1}^{\ell_2} \binom{n}{\ell_1} \frac{(\ell_1 - 1)!}{2} \cdot \ell_1 \cdot (n)_{\ell_2 - s} \cdot (n + \eta)^{- (\ell_1 + \ell_2 - s + 1)} \nonumber \\
&\leq& \sum_{\ell_1 = 3}^{\sqrt[3]{n}/2} \sum_{\ell_2 = 3}^{\ell_1} \sum_{s=1}^{\ell_2} 1/n \leq (\sqrt[3]{n}/2)^3 \cdot n^{-1} = 1/8 \,.  
\end{eqnarray}  

Let ${\mathcal F} = {\mathcal F}_1 \cup {\mathcal F}_2$. Combining~\eqref{eq::phi1} and~\eqref{eq::phi2} we conclude that $\Phi({\mathcal F}) < 1$. It thus follows by Theorem~\ref{th::ClientBES} that Client has a strategy to build a graph $G_C$ such that, if $C_1$ and $C_2$ are cycles of $G_C$, then $V(C_1) \cap V(C_2) = \emptyset$. It is easy to see that a graph with no pair of intersecting cycles is $K_4$-minor free. 
{\hfill $\Box$ \medskip\\}

\noindent \textbf{Proof of Theorem~\ref{th::WCminor2}}
Let $n, t$ and $\alpha$ be as in the statement of the theorem. If a graph admits a $K_t$-minor, it must contain at least 
\begin{equation*}
\sum_{k=3}^t \binom{t}{k} \frac{(k-1)!}{2} \geq \frac{(t-1)!}{2} \geq (c_1 t)^t \geq e^{c_2 t \log t} 
\end{equation*}
cycles, where $c_1$ and $c_2$ are positive constants. It is therefore sufficient to show that Client has a strategy to avoid building $e^{c_2 t\log t}$ cycles.

\medskip

Let $\mathcal{F} = \{E(C) : C \textrm{ is a cycle of } K_n\}$. Then
\begin{align*} 
\Phi(\mathcal{F}) &= \sum_{A \in \mathcal{F}} (q+1)^{-|A|} = \sum_{k=3}^n \binom{n}{k} \frac{(k-1)!}{2} (q+1)^{-k} \leq \sum_{k=3}^n \frac{1}{k} \left(\frac{n}{n - \alpha}\right)^k \nonumber \\
&\leq \sum_{k=1}^n \frac{e^{\alpha k/(n-\alpha)}}{k} \leq e^{(1+o(1)) \alpha} \sum_{k=1}^n \frac{1}{k} \leq e^{(1+o(1)) \alpha} (\log n + 1) < e^{c_2 t \log t} \,,   
\end{align*}
where the third inequality holds since $\alpha = o(n)$ and the last inequality holds by the assumed bounds on $t$ and $\alpha$.

It thus follows by Theorem~\ref{th::ClientBES} that Client has a strategy to avoid fully claiming $e^{c_2 t \log t}$ cycles. This concludes the proof of the theorem.
{\hfill $\Box$ \medskip\\}

In the proof of Theorem~\ref{th::CWminor}, we will make use of the following two results.
\begin{proposition} [\cite{HKSSpcm}, Lemma 4.8] \label{prop::girthMinor}
Let $G$ be a graph with average degree $2 + \alpha$, for some $\alpha > 0$, and girth $g \geq (1 + 2/\alpha) (4 \log_2 t + 2 \log_2 \log_2 t + c')$, where $c'$ is an absolute constant (i.e., independent of $t$ and $\alpha$). Then $G$ admits a $K_t$-minor.
\end{proposition}

\begin{theorem}[\cite{BHKL}, Theorem 1.3]
\label{th::connected}
For every integer $n\geq 4$, Waiter can force Client to build a connected graph when playing a $(1:q)$ Waiter-Client game on $E(K_n)$ if and only if $q\leq\lfloor n/2\rfloor -1$.
\end{theorem}

\noindent \textbf{Proof of Theorem~\ref{th::CWminor}} 
Assume first that $q \geq n/2 - 1$. We consider two cases according to the parity of $n$. If $n$ is even, then by monotonicity, and since $n/2$ is an integer, we may assume that $q = n/2 - 1$. Note that $\binom{n}{2}/(q+1) = n-1$; in particular, Waiter offers exactly $q+1$ edges in each round of the game. By Theorem~\ref{th::connected}, Waiter has a strategy to force Client to build a connected graph. Since, moreover, $e(G_C) = n-1$ at the end of the game, it must be a spanning tree which is $K_3$-minor free.   

Assume then that $n$ is odd. By monotonicity, and since $q$ is an integer, we may assume that $q = (n+1)/2 - 1$. By Theorem~\ref{th::connected}, Waiter has a strategy to force Client to build a connected graph when playing on $E(K_{n+1})$; let ${\mathcal S}$ be such a strategy. We present a strategy ${\mathcal S}'$ for Waiter to force Client to build a $K_3$-minor free graph when playing on $E(K_n)$. Waiter pretends the board is $E(K_{n+1})$, i.e., in his mind he adds an imaginary vertex and $n$ imaginary edges, and follows ${\mathcal S}$. If in some round he is instructed by ${\mathcal S}$ to offer only imaginary edges, then he pretends that he did, and then he chooses one of these edges arbitrarily and pretends that Client claimed it. If in some round he is instructed by ${\mathcal S}$ to offer at least one imaginary edge and at least one real edge (i.e., an edge which is actually on the board $E(K_n)$), then he offers only the real edges (recall that in a Client-Waiter game, Waiter is allowed to offer less board elements than his bias specifies) but pretends he offered all edges ${\mathcal S}$ instructed him to claim. In every other round he plays precisely as ${\mathcal S}$ instructs him to. Since, in his mind, Waiter followed ${\mathcal S}$ exactly and since ${\mathcal S}$ is a winning strategy for the game on $E(K_{n+1})$, in Waiter's mind, at the end of the game, Client's graph is a spanning tree of $K_{n+1}$. Hence, at the end of the game, $G_C$ is a forest and thus $K_3$-minor free.

Now, suppose that $q \leq (1/2 - \varepsilon) n$; by monotonicity we can in fact assume that $q = \lfloor (1/2 - \varepsilon) n \rfloor$. Let $\alpha$ be a constant satisfying 
$$
\left \lfloor \frac{\binom{n}{2}}{\lfloor (1/2 - \varepsilon) n \rfloor + 1} \right\rfloor \geq (1 + \alpha) n.
$$ 
Note that $\alpha \geq \frac{\varepsilon}{1 - \varepsilon}$. Let $k = \lfloor \log_3 (\alpha n/4) \rfloor$ and let ${\mathcal F}_k$ denote the family of edge-sets of all cycles of $K_n$ whose length is strictly smaller than $k$. Then
\begin{eqnarray*} \label{eq::largeGirth}
\Phi({\mathcal F}_k) &=& \sum_{A \in {\mathcal F}_k} (q+1)^{-|A|} = \sum_{s = 3}^{k-1} \binom{n}{s} \frac{(s-1)!}{2} (q+1)^{-s} < \sum_{s = 3}^{k-1} \left(\frac{n}{\lfloor (1/2 - \varepsilon) n \rfloor}\right)^s \nonumber \\
&<& \sum_{s = 3}^{k-1} 3^s < 3^k \leq \alpha n/4 \,.
\end{eqnarray*}  
Using Theorem~\ref{th::Oren} we infer that Client has a strategy to build a graph $G_C$ which contains a subgraph $H_C$ with at least $(1 + \alpha) n$ edges and fewer than $\alpha n/2$ cycles of length at most $k-1$. Deleting one edge from each such cycle, results in a graph $H$ with average degree at least $2 + \alpha$ and with girth at least $k$. Let $t$ be the largest integer for which $(1 + 2/\alpha) (4 \log_2 t + 2 \log_2 \log_2 t + c') \leq k$; it is easy to see that there exists a constant $c > 0$ such that $t \geq (\varepsilon n)^{c \varepsilon}$. It follows from Proposition~\ref{prop::girthMinor} that $H$ admits a $K_t$-minor. Clearly, $G_C$ admits the same minor.
{\hfill $\Box$ \medskip\\}

\noindent \textbf{Proof of Corollary~\ref{cor::WCplanarity}}
Assume first that $q \leq (1 - \varepsilon) n$. If $\varepsilon \geq 5 n^{- 1/4}$, then $\varepsilon^2 \sqrt{n}/5 \geq 5$ and thus it follows by Theorem~\ref{th::WCminor} that Waiter can force Client's graph to admit a $K_5$-minor; Client's graph is then non-planar. Assume then that $q \geq n + \eta$, where $\eta = \eta(n) \geq n^{2/3} \log n$. It follows by Theorem~\ref{th::WCminor} that Client has a strategy to keep his graph $K_4$-minor free. It is easy to see that every subdivision of both $K_5$ and $K_{3,3}$ admits a $K_4$-minor and thus Client's graph is planar by Kuratowski's Theorem (see, e.g.,~\cite{West}).    
{\hfill $\Box$ \medskip\\}

\noindent \textbf{Proof of Corollary~\ref{cor::CWplanarity}}
Assume first that $q \geq \lceil n/2 \rceil - 1$. It follows by Theorem~\ref{th::CWminor} that Waiter has a strategy to force Client to build a $K_3$-minor free graph; such a graph is, in particular, planar. Assume then that $q \leq n/2 - c n/ \log n$. Then $q \leq (1/2 - \varepsilon) n$, where $\varepsilon = c/ \log n$. For a sufficiently large constant $c$, it follows by Theorem~\ref{th::CWminor} that Client has a strategy to build a graph which admits a $K_5$-minor and is thus non-planar.      
{\hfill $\Box$ \medskip\\}

\section{Colorability games} \label{sec::colorability}

In this section we prove Theorems~\ref{th::kColorabilityWC} and~\ref{th::kColorabilityCW}. In both proofs we will make use of the following well-known result:

\begin{theorem} [\cite{Kim}] \label{th::girth5}
Let $G$ be a graph with maximum degree $\Delta$ and girth at least $5$. Then
$$
\chi(G) \leq (1 + \nu(\Delta)) \Delta/\log \Delta \,,
$$
where $\nu(\Delta)$ is a function which tends to zero as $\Delta$ tends to infinity. 
\end{theorem}

\bigskip

\noindent \textbf{Proof of Theorem~\ref{th::kColorabilityWC}}
Let $q \geq (8e + \alpha) n/(k \log k)$ be an integer and let $\nu$ be the function appearing in the statement of Theorem~\ref{th::girth5}. Fix an arbitrarily small constant $\varepsilon > 0$ and let $k_0$ be the smallest integer such that $\log \log k_0 \geq \log 3 - \log (1 - \varepsilon)$ and $\nu((1 - \varepsilon) k \log k/3) \leq \varepsilon$ holds for every $k \geq k_0$. Assume first that $k \geq \max \{k_0, 1000\}$. Client's strategy is based on Theorem~\ref{th::ClientBES}. In order to present it we first consider several sums. 

Let ${\mathcal F}_1 = \{E(C_1) \cup E(C_2) : C_1 \textrm{ and } C_2 \textrm{ are cycles of } K_n, \; |C_1|,|C_2|\in\{3,4\} \textrm{ and } V(C_1) \cap V(C_2) \neq \emptyset\}$. Then
\begin{eqnarray} \label{eq::intersectingTriangles}
\Phi({\mathcal F}_1) &=& \sum_{A \in {\mathcal F}_1} (q+1)^{- |A|} \leq n^4 \left(\frac{k \log k}{(8e + \alpha) n} \right)^5 + 3 n^5 \left(\frac{k \log k}{(8e + \alpha) n} \right)^6 + 2 n^6 \left(\frac{k \log k}{(8e + \alpha) n} \right)^7 \nonumber \\
&+& n^7 \left(\frac{k \log k}{(8e + \alpha) n} \right)^8 = o(1) \,. 
\end{eqnarray}   

Let ${\mathcal F}_2 = \{F : \exists S \subseteq V(K_n) \textrm{ such that } S \neq \emptyset, \; F \subseteq E_{K_n}(S) \textrm{ and } |F| = |S| k \log k/16\}$. Then
\begin{eqnarray} \label{eq::manyEdges} 
\Phi({\mathcal F}_2) &=& \sum_{A \in {\mathcal F}_2} (q+1)^{- |A|} \leq \sum_{t=1}^n \binom{n}{t} \binom{\binom{t}{2}}{t k \log k/16} (q+1)^{- t k \log k/16} \nonumber \\
&\leq& \sum_{t=1}^n \left[\frac{e n}{t} \left(\frac{8 e t}{k \log k} \cdot \frac{k \log k}{(8e + \alpha) n} \right)^{k \log k/16} \right]^t \leq \sum_{t=1}^n \left[e \left(\frac{8 e}{8e + \alpha} \right)^{k \log k/16} \right]^t \nonumber \\
&<& 1/3 \,,
\end{eqnarray}
where the third inequality holds since $k$ is assumed to be sufficiently large and the last inequality holds for an appropriately chosen $\alpha = \alpha(k)$; it is not hard to see that $\alpha$ can be chosen such that it tends to zero as $k$ tends to infinity. 

Let ${\mathcal F}_3 = \{F \cup F' : \exists S \subseteq V(K_n) \textrm{ such that } S \neq \emptyset, \; F \subseteq E_{K_n}(S), \; F' \subseteq E_{K_n}(S, V(K_n) \setminus S), \; |F| = |S| k/6 \textrm{ and } |F'| = |S| k \log k/8\}$. Then
\begin{eqnarray} \label{eq::largeMinDeg}
\Phi({\mathcal F}_3) &=& \sum_{A \in {\mathcal F}_3} (q+1)^{- |A|} \leq \sum_{t=1}^n \binom{n}{t} \binom{\binom{t}{2}}{t k/6} \binom{t(n-t)}{t k \log k/8} (q+1)^{- (t k/6 + t k \log k/8)} \nonumber \\ 
&\leq& \sum_{t=1}^n \left[\frac{e n}{t} \left(\frac{3 e t}{k} \cdot \frac{k \log k}{(8e + \alpha) n} \right)^{k/6} \left(\frac{8 e (n-t)}{k \log k} \cdot \frac{k \log k}{(8e + \alpha) n} \right)^{k \log k/8} \right]^t \nonumber \\
&\leq& \sum_{t=1}^n \left[e \left(\frac{3e}{8e + \alpha} \right)^{k/6} \left(\frac{8 e}{8e + \alpha} \right)^{k \log k/8} (\log k)^{k/6} \right]^t < 1/3 \,,
\end{eqnarray}
where the third inequality holds since $k$ is assumed to be sufficiently large and the last inequality holds for an appropriately chosen $\alpha = \alpha(k)$; it is not hard to see that $\alpha$ can be chosen such that it tends to zero as $k$ tends to infinity.

Let ${\mathcal F} = {\mathcal F}_1 \cup {\mathcal F}_2 \cup {\mathcal F}_3$. Combining~\eqref{eq::intersectingTriangles}, \eqref{eq::manyEdges} and~\eqref{eq::largeMinDeg}, it follows from Theorem~\ref{th::ClientBES} that Client has a strategy to build a graph $G_C$ which satisfies the following three properties:
\begin{description}
\item [(a)] If $C_1$ and $C_2$ are cycles of length at most $4$ in $G_C$, then $V(C_1) \cap V(C_2) = \emptyset$.
\item [(b)] $e_{G_C}(S) \leq |S| k \log k/16$ for every $S \subseteq V(K_n)$.
\item [(c)] For every $S \subseteq V(K_n)$, if $e_{G_C}(S) \geq |S| k/6$, then $e_{G_C}(S, V(K_n) \setminus S) < |S| k \log k/8$.
\end{description}

It remains to prove that a graph which satisfies Properties (a), (b) and (c), has chromatic number at most $k$. Let $X = \{u \in V(K_n) : d_{G_C}(u) \leq (1 - \varepsilon) k \log k/3\}$ and let $Y = V(K_n) \setminus X$. Let $X_1 \cup X_2$ be a partition of $X$ such that both $G_C[X_1]$ and $G_C[X_2]$ have girth at least $5$; such a partition exists by Property (a). Clearly $\Delta(G_C[X_i]) \leq \Delta(G_C[X]) \leq (1 - \varepsilon) k \log k/3$ holds for $i \in \{1,2\}$ by the definition of $X$. Since $k \geq k_0$, using Theorem~\ref{th::girth5} we infer that  
$$
\chi(G_C[X_i]) \leq (1 + \nu((1 - \varepsilon) k \log k/3)) \cdot \frac{(1 - \varepsilon) k \log k/3}{\log ((1 - \varepsilon) k \log k/3)} \leq (1 + \varepsilon) \cdot \frac{(1 - \varepsilon) k \log k/3}{\log ((1 - \varepsilon) k \log k/3)} \leq k/3 \,,
$$
holds for $i \in \{1,2\}$. Hence, $\chi(G_C[X]) \leq \chi(G_C[X_1]) + \chi(G_C[X_2]) \leq 2k/3$. 

Suppose for a contradiction that $\chi(G_C) \geq k+1$. Since $\chi(G_C) \leq \chi(G_C[X]) + \chi(G_C[Y]) \leq 2k/3 + \chi(G_C[Y])$, it follows that $\chi(G_C[Y]) \geq k/3 + 1$. Therefore, there exists a set $Z \subseteq Y$ such that $\delta(G_C[Z]) \geq k/3$, entailing $e_{G_C}(Z) \geq |Z| k/6$. It follows by Property (b) that $e_{G_C}(Z) \leq |Z| k \log k/16$. By the definition of $Y$, and since $\varepsilon$ is arbitrarily small, we then have $e_{G_C}(Z, V(K_n) \setminus Z) \geq |Z| k \log k/8$. However, this contradicts Property (c). We conclude that $\chi(G_C) \leq k$ as claimed.   

\medskip 

Assume then that $2 \leq k < \max \{k_0, 1000\}$. Let ${\mathcal F} = \{F : \exists S \subseteq V(K_n) \textrm{ such that } S \neq \emptyset, \; F \subseteq E_{K_n}(S) \textrm{ and } |F| = |S| k/2\}$. Then
\begin{eqnarray*} 
\Phi({\mathcal F}) &=& \sum_{A \in {\mathcal F}} (q+1)^{- |A|} \leq \sum_{t=1}^n \binom{n}{t} \binom{\binom{t}{2}}{t k/2} (q+1)^{- t k/2} \\
&\leq& \sum_{t=1}^n \left[\frac{e n}{t} \left(\frac{e t}{k} \cdot \frac{k \log k}{(8e + \alpha) n} \right)^{k/2} \right]^t \leq \sum_{t=1}^n \left[e \left(\frac{e \log k}{8e + \alpha} \right)^{k/2} \right]^t < 1 \,,
\end{eqnarray*}
where the third inequality holds by our assumption that $k \geq 2$ and the last inequality holds by our assumption that $k < \max \{k_0, 1000\}$ and by choosing $\alpha = \alpha(k)$ to be sufficiently large.

It thus follows from Theorem~\ref{th::ClientBES} that Client has a strategy to build a graph $G_C$ such that every subgraph of $G_C$ admits a vertex of degree at most $k-1$. Clearly such a graph is $k$-colorable.

\bigskip

Next, assume that $q \leq c n/(k \log k)$, where $c \leq \log 2/4 - \alpha$. Let ${\mathcal F}$ denote the family of edge-sets of all cliques of $K_n$ on $\lceil n/k \rceil$ vertices. Then
\begin{eqnarray}
\sum_{A \in {\mathcal F}} 2^{-|A|/(2q-1)} &\leq& \binom{n}{\lceil n/k \rceil} 2^{- \binom{\lceil n/k \rceil}{2}/(2q)} \leq \left[e k \cdot 2^{- \frac{(\lceil n/k \rceil - 1) k \log k}{4 c n}} \right]^{\lceil n/k \rceil} \notag\\ 
&\leq& \left[2^{\log_2 e + \frac{\log k}{\log 2} + \frac{k \log k}{4 c n} - \frac{\log k}{4 c}} \right]^{\lceil n/k \rceil} = o(1) \,,
\label{eqnarray1.6}
\end{eqnarray} 
where the last equality holds by our choice of $c$ and for sufficiently large $n$.

It thus follows from Theorem~\ref{th::WaiterBES} that Waiter has a strategy to ensure that $\alpha(G_C) < \lceil n/k \rceil$ will hold at the end of the game. Since any graph $G$ has an independent set of size $\lceil v(G)/\chi(G) \rceil$, we conclude that $\chi(G_C) > k$ as claimed.   
{\hfill $\Box$ \medskip\\}

\bigskip

\noindent \textbf{Proof of Theorem~\ref{th::kColorabilityCW}}
Let $q \geq (4 + \alpha) n/(k \log k)$ be an integer and let $\nu$ be the function appearing in the statement of Theorem~\ref{th::girth5}. Fix an arbitrarily small constant $\varepsilon > 0$ and let $k_0$ be the smallest integer such that $\log \log k_0 \geq \log 2 - \log (1 - \varepsilon)$ and $\nu((1 - \varepsilon) k \log k/2) \leq \varepsilon$ holds for every $k \geq k_0$. Assume first that $k \geq k_0$. We present a strategy for Waiter; it is divided into the following two stages:

\bigskip

\noindent \textbf{Stage I:} Waiter forces Client to build a graph $H_1$ of maximum degree at most $(1 - \varepsilon) k \log k/2$ and girth at least $5$ such that $d_{G_F}(u) \leq (k \log k)^3$ holds for every $u \in V(K_n)$ at the end of this stage.  

\medskip

\noindent \textbf{Stage II:} Waiter forces Client to build a linear forest $H_2 := G_C \setminus H_1$.

\bigskip

We will prove that Waiter can indeed follow the proposed strategy. First, we introduce some notation and terminology. An edge $e \in E(G_F)$ is called \emph{dangerous} if adding it to $G_C$ creates a cycle of length $3$ or $4$. Note that once an edge becomes dangerous, it remains dangerous for as long as it is free. At any point during the game, we will denote the set of dangerous edges by $D$. 

We can now describe Stage I of Waiter's strategy in more detail. In the first round, Waiter offers Client $q+1$ arbitrary edges. For every integer $i \geq 1$, let $x_i y_i$ denote the edge claimed by Client in the $i$th round, let $X_i = \{e \in E(G_F) : x_i \in e\} \setminus D$ and let $Y_i = \{e \in E(G_F) : y_i \in e\} \setminus D$. For every integer $i \geq 1$, in the $(i+1)$st round Waiter offers Client $\min \{|X_i|, \lceil (q+1)/2 \rceil\}$ arbitrary edges of $X_i$ and $\min \{|Y_i|, \lfloor (q+1)/2 \rfloor\}$ arbitrary edges of $Y_i$. If $X_i = Y_i = \emptyset$, then Waiter chooses an arbitrary edge $uv \in G_C$ for which $\{e \in E(G_F) \setminus D : \{u,v\} \cap e \neq \emptyset\} \neq \emptyset$ and plays as if Client claimed $uv$ in the $i$th round (recall that Waiter is allowed to offer fewer than $q+1$ edges per round, but he must offer at least one). Finally, if $\{e \in E(G_F) \setminus D : \{u,v\} \cap e \neq \emptyset\} = \emptyset$ for every $uv \in G_C$, then Waiter offers Client $\min \{q+1, |E(G_F) \setminus D|\}$ arbitrary edges of $E(G_F) \setminus D$; if the latter minimum is zero, then Stage I is over and Waiter proceeds to Stage II.     

Since Stage I of the proposed strategy never instructs Waiter to offer Client any dangerous edges, it is evident that Client's graph will have girth at least $5$ at the end of Stage I. Let $u$ be an arbitrary vertex of $K_n$. Every time Client claims an edge $uv$ for some $v \in V(K_n)$, Waiter responds by offering at least $\lfloor (q+1)/2 \rfloor$ edges which are incident with $u$. Since Client can claim at most one of these edges, Waiter claims at least $\lfloor (q+1)/2 \rfloor - 1$ of them. Note that the last time Waiter offers an edge which is incident with $u$ is a possible exception. Indeed, in the worst case, he might offer only one such edge which might then be claimed by Client. We conclude that, except possibly once, for every edge of $\{uw : w \in V(K_n)\}$ which Client claims in Stage I, Waiter claims at least $\lfloor (q+1)/2 \rfloor - 1$. Therefore, $d_{H_1}(u) \leq \left \lceil \frac{n-1}{\lfloor (q+1)/2 \rfloor} \right \rceil + 1 \leq (1 - \varepsilon) k \log k/2$, where the last inequality holds if $\alpha$ is chosen to be sufficiently large compared to $\varepsilon$. Since $u$ was arbitrary, we conclude that $\Delta(H_1) \leq (1 - \varepsilon) k \log k/2$ holds at the end of Stage I. 

At the end of Stage I, fix some vertex $u \in V(K_n)$ and let $v \in V(K_n)$ be such that $uv \in D$. It follows that there exists a vertex $z \in V(K_n)$ such that $uz, zv \in E(H_1)$ or vertices $x, y \in V(K_n)$ such that $ux, xy, yv \in E(H_1)$. That is, there is a path of length two or three between $u$ and $v$ in $H_1$. Since the number of paths of length $\ell$ in $H_1$, starting at $u$, is at most $\Delta(H_1)^{\ell}$, we conclude that $|\{e \in D : u \in e\}| \leq (k \log k/2)^2 + (k \log k/2)^3 \leq (k \log k)^3$ holds for every $u \in V(K_n)$ as claimed. 

Next, we prove that Waiter can play according to Stage II of the proposed strategy. For every integer $i \geq 1$, let $x_i y_i$ denote the edge Client claims in the $i$th round of Stage II. In the first round of this stage, Waiter identifies an inclusion maximal set $A_1 \subseteq V(K_n)$ such that the number of free edges with at least one endpoint in $A_1$ is at most $q+1$ and offers Client all of these edges. Note that this edge set is non-empty since $q \gg (k \log k)^3 \geq \Delta(G_F)$. Assume that $t$ rounds were already played in Stage II for some $t \geq 1$. In the $(t+1)$st round of this stage, Waiter identifies an inclusion maximal set $A_{t+1} \subseteq V(K_n)$ such that $\{x_t, y_t\} \subseteq A_{t+1}$ and the number of free edges with at least one endpoint in $A_{t+1}$ is at most $q+1$ and offers Client all of these edges; again, this set is non-empty since $q \gg (k \log k)^3 \geq \Delta(G_F)$. Note that this strategy ensures that, for every $i \geq 1$, $d_{G_F}(x_i) = 0$ or $d_{G_F}(y_i) = 0$ holds immediately after the $i$th round of Stage II. Since, moreover, if there are still free edges incident with $\{x_i, y_i\}$, Waiter offers all of them  in round $i+1$, it follows that the graph $H_2 \subseteq G_C$ Client builds in Stage II is indeed a linear forest.      

It remains to prove that, by following the proposed strategy, Waiter forces Client's graph to be $k$-colorable. It follows from Theorem~\ref{th::girth5} that $\chi(H_1) \leq (1 + \nu((1 - \varepsilon) k \log k/2)) \cdot \frac{(1 - \varepsilon) k \log k/2}{\log ((1 - \varepsilon) k \log k/2)} \leq (1 + \varepsilon) \cdot \frac{(1 - \varepsilon) k \log k/2}{\log ((1 - \varepsilon) k \log k/2)} \leq k/2$, where the second and third inequalities follow by our choice of $k_0$. Moreover, it is evident that $\chi(H_2) \leq 2$. We conclude that $\chi(G_C) = \chi(H_1 \cup H_2) \leq \chi(H_1) \chi(H_2) \leq k$. 

\medskip

Assume then that $2 \leq k < k_0$. In a similar way to Stage I above, for sufficiently large $\alpha = \alpha(k)$, Waiter can limit the maximum degree in Client's graph at the end of the game to $k-1$. It readily follows that $\chi(G_C) \leq k$.  

\bigskip

Next, assume that $q \leq c n/(k \log k)$, where $c \leq \log 2/2 - \alpha$. Let ${\mathcal F}$ denote the family of edge-sets of all cliques of $K_n$ on $\lceil n/k \rceil$ vertices. Then 
$$
\sum_{A \in \mathcal{F}} \left(\frac{q}{q+1}\right)^{|A|} \leq \sum_{A \in {\mathcal F}} 2^{-|A|/q} = o(1),
$$
where the last equality holds by a calculation similar to~\eqref{eqnarray1.6}.

Thus, by Theorem \ref{th::ClientBESpreventing}, Client has a strategy to claim an edge in every clique of size $\lceil n/k \rceil$ in $K_n$. In particular, this means that $\alpha(G_C) = \omega(G_W) < \lceil n/k \rceil$ at the end of the game. Since any graph $G$ has an independent set of size $\lceil v(G)/\chi(G) \rceil$, we conclude that $\chi(G_C) > k$ as claimed.
{\hfill $\Box$ \medskip\\}

\section{Concluding remarks and open problems}
\label{sec::openprob}

\noindent \textbf{Sharper bounds for colorability games}. We proved that there are absolute constants $c_1 \geq c_2 > 0$ such that, for every $k \geq 2$ and sufficiently large $n$, the threshold bias of the Waiter-Client non-$k$-colorability game $(E(K_n), \mathcal{NC}_k)$ is between $(c_1 n)/(k \log k)$ and $(c_2 n)/(k \log k)$ (see Theorem~\ref{th::kColorabilityWC}). The analogous result for Client-Waiter games was proved as well (see Theorem~\ref{th::kColorabilityCW}). This matches quite well the probabilistic intuition~\cite{AN} and the corresponding known results for Maker-Breaker games~\cite{HKSSpcm}. For sufficiently large $k$, the constants $c_1$ and $c_2$ are not too far apart but, in contrast to our very precise results for the non-planarity and $K_t$-minor games, we have not determined the asymptotic value of the threshold bias of the non-$k$-colorability game for any $k \geq 2$. When trying to prove such a result, a good place to start might be the \emph{non-bipartite game} $(E(K_n), \mathcal{NC}_2)$. For the Maker-Breaker version of this game, the threshold bias is known to be between $\lceil n/2 \rceil$ (see~\cite{BPcycles}) and $(1 - 1/\sqrt{2} - o(1)) n$ (see~\cite{BP}). 

For the Waiter-Client game $(E(K_n), \mathcal{NC}_2)$, we believe that the threshold bias is $(1 + o(1)) n$. Indeed, for $q \leq (1 - o(1)) n$, it is not hard to devise an explicit strategy for Waiter to force an odd cycle in Client's graph (for example, as in the proof of Claim~\ref{cl::longPath}, Waiter can force Client to build a path $P$ of length $(2 + o(1)) \sqrt{q}$ such that there are at least $q+1$ free edges with both endpoints in $P$, each closing an odd cycle in Client's graph; in a single round he then offers any $q+1$ of these edges). On the other hand, it was conjectured in~\cite{BHKL} that Client has a strategy to avoid any cycle if $q \geq (1 + o(1)) n$. If true, this will prove that the threshold bias of the Waiter-Client game $(E(K_n), \mathcal{NC}_2)$ is indeed $(1 + o(1)) n$. A direct application of Theorem~\ref{th::ClientBES} shows that Client can avoid odd cycles if $q \geq (1 + \alpha) n$, where $\alpha = (1 - \tanh(2))/\tanh(2) \approx  0.0374$.               

For the Client-Waiter game $(E(K_n), \mathcal{NC}_2)$, we believe that the threshold bias is $(1/2 + o(1)) n$. Indeed, it immediately follows from Theorem~\ref{th::CWminor}, that Waiter can force Client to build a bipartite graph if $q \geq \lceil n/2 \rceil - 1$. However, the best we can currently prove in the opposite direction is that Client can build an odd cycle if $q \leq (1/(4 \log 2) - o(1)) n$. This is done by applying Theorem~\ref{th::ClientBESpreventing} to the family $\{E(Q_1) \cup E(Q_2) : Q_1 \textrm{ and } Q_2 \textrm{ are cliques, } V(Q_1) \cap V(Q_2) = \emptyset \textrm{ and } V(Q_1) \cup V(Q_2) = V(K_n)\}$. 

\noindent \textbf{Minor games in the critical window}. For a graph $G$, let $ccl(G)$ denote the order of the largest complete minor in $G$. Our results for the Waiter-Client $K_t$-minor game exhibit a very strong probabilistic intuition. In both the \emph{sub-critical regime} (that is, when Waiter's bias $q$ is at least $(1 + \varepsilon) n$ and, correspondingly, the edge probability in the random graph is at most $(1 - \varepsilon) /n$) and the \emph{super-critical regime} (that is, when Waiter's bias $q$ is at most $(1 - \varepsilon) n$ and, correspondingly, the edge probability in the random graph is at least $(1 + \varepsilon) /n$), we proved that a.a.s. $ccl(G_C) = \Theta(ccl(G(n, 1/(q+1))))$. The graph invariant $ccl(G(n,p))$ has also been investigated in the \emph{critical window}, that is, when $p$ is very close to $1$. Building on results of \L uczak~\cite{L1, L2}, it was proved by Fountoulakis, K\"uhn and Osthus~\cite{FKOregular} that if $p = 1/n + \lambda n^{- 4/3}$, where $1 \ll \lambda \ll n^{1/3}$, then a.a.s. $ccl(G(n,p)) = \Theta(\lambda^{3/2})$. It would be interesting to know whether analogous results hold in the game setting as well. Note that Theorem~\ref{th::WCminor} does provide a non-trivial lower bound on $ccl(G_C)$ when $n^{- 1/4} \ll \varepsilon \ll 1$. Indeed, if $n^{- 1/4} \ll \varepsilon \ll 1$ is chosen such that $\lfloor \binom{n}{2}/((1 - \varepsilon) n + 1) \rfloor = (1/n + \lambda n^{- 4/3}) \binom{n}{2}$, then $1 \ll \lambda \ll n^{1/3}$ and $ccl(G_C) \geq c \lambda^2 n^{- 1/6}$ for some constant $c > 0$. However $\lambda^2 n^{- 1/6} \ll \lambda^{3/2}$ whenever $\lambda \ll n^{1/3}$ and, moreover, $\lambda^2 n^{- 1/6} \ll 1$ if $\lambda \ll n^{1/12}$.              

On the other hand, even if $q = (1/2 - \varepsilon) n$, where $\varepsilon > 0$ is arbitrarily small but fixed, the largest complete minor we proved Client can guarantee in his graph in the Client-Waiter $K_t$-minor game is of order $n^{\gamma}$, where $\gamma = \gamma(\varepsilon) > 0$ is a small constant. It would be interesting to know whether Client can ensure a larger minor, ideally of order $\Theta(\sqrt{n})$, as in the Waiter-Client version of the game.

\section*{Acknowledgment}

We would like to thank Oren Dean for providing us with Theorem~\ref{th::Oren} and its proof, and the anonymous referees for helpful comments.

\bibliographystyle{amsplain}

\end{document}